\documentclass[12pt]{article}

\usepackage{lscape}
 \usepackage[ansinew]{inputenc}
 \usepackage{multicol}
 \usepackage[dvips]{graphicx}
\usepackage[psamsfonts]{amssymb}
 \usepackage[all]{xy}

\usepackage[psamsfonts]{eucal}
\usepackage{amssymb, amsmath}
\usepackage{amsthm}
\usepackage{geometry}
\usepackage{enumerate}
\usepackage{float}

\usepackage{color}
\def\red{\textcolor{red}}
\theoremstyle{plain}

\newtheorem{proposition}{Proposition}[section]
\newtheorem{theoremb}[proposition]{Theorem}

\theoremstyle{definition}
\newtheorem{definition}[proposition]{Definition}
\newtheorem{example}[proposition]{Example}

\theoremstyle{remark}

\def\cG{{\mathcal G}}

\newcommand{\bz}{\mathbb Z}
\newcommand{\bq}{\mathbb Q}
\def\hL{{\widehat{\mathbb L}}}

\def\ad{{\rm ad}}

    \newcommand{\lasu}{{\mathfrak{L}}}

 \newcommand{\lib }{\mathbb{L}}

 \newcommand{\MC}{\operatorname{{\rm MC}}}
\newcommand{\mc}{{\MC}}

   \newcommand{\libc}{{\widehat\lib}}

\title{The Deligne groupoid of the Lawrence-Sullivan interval}

\author{Urtzi Buijs, Yves F\'elix, Aniceto Murillo and Daniel Tanr\'e\footnote{The first author has been partially supported by the Marie Curie COFUND programme U-mobility, co-financed by the University of M\'alaga, the European Commision FP7 under GA No. 246550, and MINECO (COFUND2013-40259).\hfill\break
\indent The first and third authors have been partially supported by the Junta de Andaluc{\'\i}a grant FQM-213.\hfill\break\indent The  authors are partially supported by the MINECO grant MTM2013-41768-P.}}

\begin{document}

\maketitle

\begin{abstract}
In this paper we completely describe the Deligne groupoid of the Lawrence-Sullivan interval as two parallel rational lines.\end{abstract}

\section*{Introduction}
Let $\mc(L)$ be the set of Maurer-Cartan elements of  a differential graded Lie algebra $L$ over $\bq$ which is assumed  the base field henceforth. The group  $L_0$ endowed with the Baker-Campbell-Hausdorff product, acts on $\mc(L)$ as ``a group of gauge transformations on flat connections" (see  next section for precise and explicit terms). The groupoid associated to this action, known as the {\em Deligne groupoid}, was first introduced in \cite{goldmill} as a fundamental object to understand the Deligne principle by which every deformation functor is governed precisely by such a groupoid. See also \cite{getz,hi}.

On the other hand, \emph{the Lawrence-Sullivan interval} \cite{LS} is a complete differential free graded Lie algebra $\lasu=(\widehat\lib(a,b,x),\partial)$ generated by two Maurer-Cartan elements $a,b$ and by an element $x$ of degree $0$ joining $a$ and $b$ via the above action (see also next section). This particular object plays an essential role on the topological realization of (complete) differential graded Lie algebras \cite{bumu2,bufemutan} as well as on their homotopical behaviour \cite{bumu1,patan}.

In this note, we explicitly describe the Deligne groupoid of $\lasu$ and prove the following (see Theorem \ref{yanose} for a precise statement):

\begin{theoremb}\label{main} The deligne groupoid of $\lasu$ is isomorphic to two disjoint copies of the rationals.
\end{theoremb}

This readily implies that its nerve is a simplicial set homotopy equivalent to two disjoint points, which retrieves by a different argument the known geometrical realization of the Lawrence-Sullivan interval \cite[Ex. 5.6]{bumu2}, \cite[\S4]{bufemutan}.

As another consequence we also obtain that any perturbation of   $\lasu$ produces an isomorphic DGL.

\begin{theoremb}\label{main2} Let $z\in\mc(\lasu)$. Then, $(\widehat\lib(a,b,x),\partial_z)$ in which $\partial_z=\partial+\ad_z$, is isomorphic to $\lasu$.
\end{theoremb}

\section{Preliminaries}

Recall that a  {graded Lie algebra} is a $\mathbb Z$-graded vector space $L=\oplus_{p\in\bz}L_p$ endowed with a bilinear  { Lie bracket} satisfying  antisymmetry  $[x,y] = -(-1)^{\vert x\vert \vert y\vert} [y,x]$ and Jacobi identity
$$(-1)^{\vert x\vert \vert z\vert} [x,[y,z]] + (-1)^{\vert y\vert \vert x\vert} [y, [z,x]] + (-1)^{\vert z\vert  \vert y\vert} [z,[x,y]]  = 0\,.$$
Here $\vert x\vert$ denotes the degree of $x$.
A \emph{differential graded Lie algebra} or DGL  is a graded Lie algebra $L$ together with a linear derivation $\partial$  of degree $-1$ such that $\partial^2= 0$.

A Maurer-Cartan element of a given DGL is an element $z\in L_{-1}$ satisfying $\partial z+ \frac{1}{2}[z,z] = 0$. We denote by ${\MC}(L)$ the set of Maurer-Cartan elements. These are preserved by DGL morphisms.  Given  $z\in {\MC}(L)$, the perturbed derivation $\partial_z=\partial+\ad_z$ is again a differential on $L$.

The \emph{completion} $\widehat{L}$ of a graded Lie algebra $L$ is the projective limit
$$\widehat{L} = \varprojlim_n L/L^n$$
where $L^1 = L$ and for $n\geq 2$, $L^n = [L, L^{n-1}]$. A Lie algebra $L$ is called \emph{complete} if $L$ is isomorphic to its completion. The completion of the free lie algebra generated by the graded vector space $V$ is denoted by $\widehat\lib(V)$.

Given $L$ a complete DGL, the {\em gauge action} of $L_0$ on ${\MC}(L)$ determines an equivalence relation among Maurer-Cartan elements defined as follows (see for instance \cite[\S4]{mane} or Proposition \ref{despues} below): given $x\in L_0$ and $z\in {\MC}(L)$,
$$x\,\cG\, z=e^{\ad_x}(z)-\frac{e^{\ad_x}-1}{\ad_x}(\partial x).$$
Here and from now on, the integer $1$ inside an operator will denote the identity. Explicitly,
$$
x\,\cG\, z=\sum_{i\ge0}\frac{\ad_x^i(z)}{i!}-\sum_{i\ge0}\frac{\ad_x^{i}(\partial x)}{(i+1)!}.
$$
The \emph{Deligne groupoid} of $L$ has  $\mc(L)$ as objects, and elements $x\in L_0$ as arrows from   $x\,\cG\,z$  to $z$.
Geometrically \cite{kont,LS}, interpreting Maurer-Cartan elements as points in a space, one thinks of $x$ as a flow taking   $x\,\cG\, z$ to $z$ in unit time. In topological terms, the points $z$ and $x\,\cG\, z$ are in the same path component.

 \begin{definition}\label{def:LSconstruction}\label{sec:LScylinder}\cite{LS}
\emph{The Lawrence-Sullivan interval} is the complete free DGL
$${\lasu} = (\widehat{\mathbb L} (a,b,x),\partial),$$
in which $a$ and $b$ are Maurer-Cartan elements, $x$ is of degree~0 and
\begin{equation}\label{equa:dxLSinterval}
\partial x = \ad_xb + \sum_{n=0}^\infty \frac{B_n}{n!}\,  \ad_x^n (b-a) = \ad_xb + \frac{\ad_x}{e^{\ad_x}-1} (b-a)\,
\end{equation}
 where $B_n$ are the Bernoulli numbers. Using the identity
$$
\left( \frac{-x}{e^{-x}-1}\right) = x+ \left( \frac{x}{e^x-1}\right),
$$
we may also write
\begin{equation}\label{equa:dxLSinterval2}
\partial x =  \ad_xa + \frac{\ad_{-x}}{e^{-\ad_x}-1} (b-a).
\end{equation}
Moreover, the differential $\partial$ in $\lasu$ is the only one for which $a$ and $b$ are Maurer-Cartan elements and either, its linear part $\partial_1$ satifies $\partial_1 x=b-a$ \cite[Thm. 1.4]{bufemutan}, or else $x\,\cG\,b=a$ \cite[Thm. 1]{LS}, \cite{patan}.

The {\em change of orientation} of $\lasu$ is the automorphism
$$
\gamma\colon\lasu \stackrel{\cong}{\longrightarrow}\lasu,\qquad\gamma(a)=b,\quad\gamma(b)=a,\quad\gamma(x)=-x.
$$
In fact, using (\ref{equa:dxLSinterval}) and (\ref{equa:dxLSinterval2}), one easily sees that $\gamma$ commutes with the differential.
\end{definition}

To be self-contained we give  a short proof, using the LS-interval, of the fact that the gauge action preserves Maurer-Cartan elements.

\begin{proposition}\label{despues} Let $(L,d)$ be a complete DGL, $z\in {\MC}(L)$ and $x\in L_0$. Then, $x\, \cG\, z$ is a Maurer-Cartan element.
\end{proposition}

\begin{proof} Write $c=x\,\cG\, z$. The formula for the gauge action  can be written
$$dx= \ad_xz + \frac{{\ad}_x}{e^{{\ad}_x}-1} (z- c).$$
Then $d^2x= 0$ implies that $w= dc+\frac{1}{2}[c,c]$ belongs to the ideal generated by $x$. Suppose $w\neq 0$ then $w= w_r+w_r'$ with $w_r$ a non zero linear combination of Lie brackets containing $r$ times $x$,
and $w_r'$ a linear combination of Lie brackets containing at least $r+1$ times $x$. If we denote by $\partial$ the usual differential on the LS-interval on $ \widehat{\mathbb L}(c,z,x)$, then $dx=\partial x$, $dz=\partial z$. From $d^2=0$ and $\partial^2=0$ we deduce
$$0 = -w+ \frac{1}{2}[x,w] - \sum_{n\geq 2} \frac{B_n}{n!} {\ad}_x^n (w).$$
This implies that $w_r= 0$.  Therefore $w = 0$ and $c=x\, \cG\,z$ is a Maurer-Cartan element.\end{proof}

\section{The Deligne groupoid of $\lasu$}

We first recall the Lawrence-Sullivan model   of the subdivision of the interval. Let $(\libc(a_0,a_1, a_2,x_1,x_2),\partial)$ be two glued LS-intervals. That is, $a_0, a_1$ and $a_2$ are Maurer-Cartan elements, $\partial x_1 = \ad_{x_1}a_1 + \frac{ \ad_{x_1}}{e^{\ad_{x_1}}-1} (a_1-a_0)$ and
  $\partial x_2 = \ad_{x_2}a_2 + \frac{ \ad_{x_2}}{e^{\ad_{x_2}}-1} (a_2-a_1)$.

 \begin{theoremb} \label{prop:subdivision} {\rm \cite[Thm. 2]{LS}}
 The map
 $$\gamma\colon (\hL(a,b,x),\partial) \longrightarrow(\hL(a_0,a_1,a_2, x_1, x_2),\partial),$$
$\gamma(a)= a_0$, $\gamma(b)= a_2$ and $p(x) = x_1*x_2$, is a DGL morphism.
% This property is equivalent to
%\begin{equation}\label{equa:subdivision}
%d (x_{1}*x_{2})=\ad_{x_{1}*x_{2}}(a_{2})+\frac{\ad_{x_{1}*x_{2}}}{e^{\ad_{x_{1}*x_{2}}}-1}(a_{2}-a_{0}).
%\end{equation}
 \end{theoremb}
 Here, $*$ denotes the Baker-Campbell-Hausdorff product. This theorem translates to the statement
 $$(x_1*x_2)\, \cG\, a_2 = a_0 = x_1\, \cG\, (x_2\, \cG\, a_2),$$
which reproves that the gauge action is in fact a group action of $(L_0, *)$ on ${\MC}(L)$.

\begin{proposition}\label{unico} A Maurer-Cartan element of $\lasu$ is completely determined by its linear part.
\end{proposition}

\begin{proof}  Suppose $\omega$ and $\omega'$ are Maurer-Cartan elements with the same linear part. We write $\omega = \sum_{i\geq 1} \omega_i$ and $\omega' = \sum_{i\geq 1}\omega_i'$, where $\omega_i$ and $\omega_i'$ belong to the words of length $i$. Suppose by induction that $\omega_i= \omega_i'$ for $i\leq r$. We write
$$
\begin{array}{l}
\omega_{r+1} = \lambda_r\, {\ad}^r(x)(a) + \mu_r\, {\ad}^r(x)(b),\\
\omega_{r+1}' = \lambda_r'\, {\ad}^r(x)(a) + \mu_r'\, {\ad}^r(x)(b),
\end{array}$$
with $\lambda_r, \mu_r, \lambda_r', \mu_r'\in \mathbb Q$,
and denote by $K$ the ideal generated by the words containing at least $r$ times the variable $x$. Since $\omega$ and $\omega'$ are Maurer-Cartan elements, $$\partial(\omega -\omega') = -\frac{1}{2} \left( [\omega, \omega]-[\omega', \omega']\right)\in K.$$ Next, a direct computation shows that, modulo $K$ and in the universal enveloping algebra $U\lasu$,
$$
\partial(\ad_x^ra)=(b-a)ax^{r-1}+x\Gamma_a,\quad \Gamma_a\in U\lasu.
$$
From here we deduce that, also modulo $K$ and in $U\lasu$,
$$\partial\omega - \partial\omega' = (\lambda_r-\lambda_r')(b-a)ax^{r-1} + (\mu_r-\mu_r')(b-a)bx^{r-1} + x\Gamma,\quad\Gamma\in \lasu.
 $$
 Therefore, $\lambda_r= \lambda_r'$ and $\mu_r= \mu_r'$.
\end{proof}

We now prove Theorem \ref{main}:
\begin{theoremb}\label{yanose}
 The set of Maurer-Cartan elements in $\lasu$ decomposes into two families:

\begin{itemize}

\item[(I)] Maurer-Cartan elements of the form $\lambda a+ (1-\lambda )b + \Omega$, where $\lambda\in \mathbb Q$ and $\Omega$ is a decomposable element.

\item[(II)]   Maurer-Cartan elements of the form $\mu(a-b) + \Omega$, where $\mu \in \mathbb Q$ and $\Omega$ is a decomposable element.

    \end{itemize}

Moreover, two Maurer-Cartan elements $u$ and $v$ are gauge equivalent if and only if they belong to the same family.
\end{theoremb}

\vskip 0.35cm
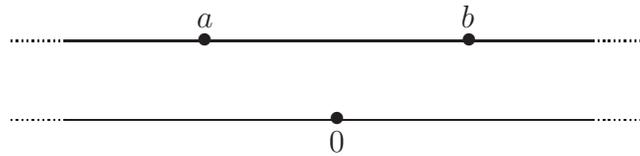
\begin{picture}(0,0)(-100,0)
\qbezier(0,0)(100,0)(200,0) \qbezier[10](200,0)(210,0)(220,0) \qbezier[10](-20,0)(-10,0)(0,0)
\put(50,-2.5){$\bullet$}\put(150,-2.5){$\bullet$}
\put(50,5){$a$}\put(150,5){$b$}

\qbezier(0,-30)(100,-30)(200,-30) \qbezier[10](200,-30)(210,-30)(220,-30)\qbezier[10](-20,-30)(-10,-30)(0,-30)
\put(100,-32.5){$\bullet$}
\put(100,-42){$0$}
\end{picture}

\vskip 1.75cm

\begin{proof}  Given $\lambda \in \mathbb Q$,  $(\lambda x)\,{\cal G}\,b  $
  is   a Maurer-Cartan element whose linear part is $\lambda a + (1-\lambda) b$.
On the other hand, for $\mu\in \mathbb Q$,   the linear part of $(\mu x) \,{\cal G}\, 0$ is $\mu a -\mu b$. This proves the existence claim and now we show that a Maurer-Cartan $\omega$ always belongs to one of the two families. Write $$\omega = \lambda a+ \mu b + \alpha [x,a] + \beta [x,b] + \omega',\quad \lambda,\mu,\alpha,\beta\in\bq$$
with $\omega'$ in the words of length at least three. Then\red{,} looking at the coefficients of $[a,a], $[b,b] and $[a,b]$ in $\partial \omega + \frac{1}{2}[\omega, \omega]$, we find
$$\left\{
\begin{array}{l} \lambda + 2 \alpha = \lambda^2,\\
\mu - 2\beta = \mu^2,\\
2(\beta-\alpha) = 2\lambda \mu.\end{array}\right.$$
Then, $(\lambda + \mu)^2= \lambda + \mu$, so $\lambda + \mu = 1$ or $\lambda + \mu = 0$, giving respectively the families (I) and (II).

We now prove the second statement. By Proposition \ref{unico}, the only Maurer-Cartan element with linear part $\lambda a+(1-\lambda)b$ is $(\lambda x)\,{\cal G}\,b$ which is gauge equivalent to $b$ by definition. Therefore,  any two elements of the family (I) are gauge equivalent.
In the same way, $(\mu x)\,{\cal G}\, 0$ is the only Maurer-Cartan element whose linear part is $\mu(a-b)$ and it is gauge equivalent to $0$ by definition. Hence, any two elements of the family (II) are gauge equivalent.

Finally suppose that we can connect an element of the family (I) to an element of the family (II). By subdivision as in Theorem \ref{prop:subdivision} this implies the existence of a DGL morphism $f\colon {\lasu} \to {\lasu}$
with $f(a)=a$, $f(b) = 0$ and $f(x) = \lambda x$,
for some $\lambda$. But then the linear part of $f\partial x-\partial fx$ is nonzero which is impossible.

\end{proof}

\begin{example} In the family (II), the Maurer-Cartan elements corresponding to $a-b$ and $b-a$ are respectively $a - e^{\ad_x}(b)$ and $b - e^{-\ad_x}(a)$. In fact, taking into account the identity
 $$\displaystyle{\frac{e^{-t}-1}{-t}\cdot \frac{t}{e^t-1} = e^{-t}},$$
and  applying $\displaystyle{\frac{e^{-\ad_x}-1}{-\ad_x}}$ to  equation
 (\ref{equa:dxLSinterval}) we obtain
$$
\frac{e^{-\ad_x}-1}{-\ad_x} (\partial x) = b - e^{-\ad_x}(a).
$$
Changing the orientation, we get an equivalent identity,
\begin{equation}\label{ultima}
-\frac{e^{\ad_x}-1}{\ad_x} (\partial x) = a - e^{\ad_x}(b).
\end{equation}
These formulae readily implies that
$$
(-x)\,\cG\, 0=b-e^{-\ad_x}(a)\qquad\text{and}\qquad x\,\cG\, 0=a-e^{\ad_x}(b),
$$
and therefore,  they are both Maurer-Cartan elements of $\lasu$ with the required linear part.
\end{example}

\begin{proposition}\label{yanoses} Let $(L,d)$ be a DGL, $z,z'\in {\MC}(L)$ and $y\in L_0$. Then, the following assertions are equivalent:
  \begin{enumerate}
  \item[(1)] $y\, \cG\, z= z'$.
  \item[(2)] There exists a  DGL morphism, $\varphi : \lasu\to (L,d)$ with $\varphi (a) = z'$, $\varphi (b)= z$ and $\varphi (x) = y$.
  \end{enumerate}\end{proposition}

  \begin{proof} This follows directly from writing
  $$dy= \ad_yz + \frac{{\ad}_y}{e^{{\ad}_y}-1}(z-z').$$
  \end{proof}

We now prove Theorem \ref{main2}.

\begin{theoremb} Every perturbed LS-interval $(\widehat{\mathbb L}(a,b,x), \partial_z)$, with $z\in\mc(\mathfrak L)$, is  isomorphic to $\mathfrak L$.
\end{theoremb}

\begin{proof}If the linear part of $z$ is $\lambda( a - b)$, then for each Maurer-Cartan element $c$ in ${\mathfrak L}$, $c-z$ is a Maurer-Cartan element for $\partial_z$. We then get an isomorphism
$ f\colon ({\mathfrak L}, \partial) \stackrel{\cong}{\longrightarrow} ({\mathfrak L}, \partial_z)$
defined by $f(a) = a-z$, $f(b)=b-z$ and $f(x) = x$.

If $z=(\lambda x)\, \cG\, b$ and $\lambda \neq 0$, then by Proposition \ref{yanoses}, there is an isomorphism $f\colon  {\mathfrak L} \to {\mathfrak L}$
with $f(a) =z$, $f(b)=b$ and $f(x) = \lambda x$. In particular $f$ induces an isomorphism $ ({\mathfrak L}, \partial_a) \stackrel{\cong}{\longrightarrow} ({\mathfrak L}, \partial_z)$.
When $\lambda = 0$, the change of orientation provides an isomorphism $\gamma \colon ({\mathfrak L}, \partial_a)\stackrel{\cong}{\longrightarrow} ({\mathfrak L}, \partial_b)$.

Finally, we show that $({\mathfrak L}, \partial_a)$ is isomorphic to $({\mathfrak L}, \partial)$.
Observe that $-a$ and $b-a$ are Maurer-Cartan elements in $(\widehat{\mathbb L}(a,b,x),\partial_a)$. Write $c= -a$ and $w= b-a$ so that $(\widehat{\mathbb L}(a,b,x),\partial_a)\cong(\widehat{\mathbb L}(c,w,x),\partial')$ with $c$ and $w$ Maurer-Cartan elements and $\partial' x = \frac{-{\ad}_x}{e^{-{\ad}_x}-1}(w)$.
Remark now that the morphism of Lie algebras
$$\varphi \colon (\widehat{\mathbb L}(c,w,x), \partial') \to \lasu$$
defined by $\varphi (c)= a$, $\varphi (w) = a-e^{{\ad}_x}b = x\, \cG\, 0$ and $\varphi (x) = -x$ is a DGL isomorphism. This implies the existence of an isomorphism $(\lasu, \partial_a) \cong (\lasu, \partial)$.
\end{proof}

We finish by presenting another isomorphism class of the LS-interval: consider the quotient of $\lasu$ by the ideal generated by $a$ to obtain $(\widehat\lib(b,x),\partial')$ where
$$
\partial'x=\ad_x(b)+\frac{\ad_x}{e^{\ad_x}-1} (b)=\frac{-\ad_x}{e^{-\ad_x}-1}(b).
$$
The coproduct of this quotient with $\lib(a)$ gives back the complete Lie algebra $\widehat\lib(a,b,x)$ with a different, not perturbed, differential $\partial'$.

\begin{proposition}\label{nueva}
The LS-model $\lasu$ and  $(\widehat\lib(a,b,x),\partial')
$ are isomorphic.
\end{proposition}

\begin{proof}

Define,
$$
\varphi\colon (\widehat\lib(a,b,x),\partial')\longrightarrow \lasu
$$
by $\varphi(a)=a$, $\varphi(b)=a-e^{\ad_x}(b)$ and $\varphi(x)=-x$. This is clearly an isomorphism of Lie algebras which commutes with differentials on $a$ and $b$. Finally,
$$
\varphi(\partial'x)=\frac{-\ad_x}{e^{-\ad_x}-1}\bigl(a-e^{\ad_x}(b)\bigr)
$$which equals $\partial (\varphi x)=-\partial x$ by the  identity  (\ref{ultima}).

\end{proof}

\end{document}